\documentclass{amsart}
\usepackage{graphicx}
\usepackage{amsthm}
\usepackage{amsmath}
\usepackage[utf8]{inputenc}
\usepackage{amsfonts} 
\usepackage[english]{babel}
\usepackage{mathrsfs}
\usepackage{hyperref}
\newtheorem{Theorem}{Theorem}[section]
\newtheorem*{Theorem*}{Theorem}
\newtheorem{Proposition}{Proposition}[section]
\newtheorem*{Proposition*}{Proposition}
\newtheorem{Lemma}{Lemma}[section]
\newtheorem{Definition}{Definition}[section]
 
\newtheorem{Remark}{Remark}[section]

\newtheorem{Example}{Example}

\newtheorem{Conjecture}{Conjecture}[Theorem]
\newtheorem*{Conjecture*}{Conjecture}

\usepackage{xcolor}

\title{Bottlenecking In Graphs and a Coarse Menger-Type Theorem}

\author{Michael Bruner}
\address{University of Montana, 32 Campus Drive
Missoula, MT 59812}
\email{michael@leebruner.com}

\author{Atish Mitra}
\address{Montana Technological University, 1300 West Park Street Butte, MT 59701}
\email{amitra@mtech.edu}

\author{Heidi Steiger}
\address{Montana Technological University, 1300 West Park Street Butte, MT 59701}
\email{hsteiger@mtech.edu}

\thanks{The authors would like to thank Calvin University for their hospitality during the Workshop in Geometric Topology 2024, where part of this manuscript was completed.}

\begin{document}
\begin{abstract}
    We expand upon the notion of bottlenecking introduced in \cite{BMS24}, characterizing a spectrum of graphs and showing that this spectrum naturally extends to a concept of coarse bottlenecking. Georgakopoulos and Papasoglu \cite{GP23}, as well as Albrechtsen, Huynh, Jacobs, Knappe, and Wollan \cite{STR23}, independently proposed a Coarse Menger Conjecture  which  was recently disproved by a counterexample \cite{NSS24}. We show how the notion of bottlenecking provides a different approach to coarsening  connectedness in a graph in the form of a Coarse Menger-type theorem. We also propose a  coarse Erd\H{o}s-Menger-type  Conjecture, in the spirit of the  Erd\H{o}s-Menger conjecture which was proved after decades by Aharoni and Berger \cite{AB09}.
\end{abstract}

\maketitle

\section{Introduction} 
Menger's Theorem \cite{Men27} in classical graph theory characterizes connectivity in graphs in a natural way. Roughly, it shows that the maximum number of disjoint paths between two sets is the minimum number of cuts required to intersect every path. One rigorous formulation of this theorem is that, for any graph $G$ with subsets $X$ and $Y$ the maximum number of internally disjoint $X,Y$ paths is equal to the minimum number of vertices in a set $S$ such that every $X,Y$ path must intersect $S$. A natural coarsening of this idea was recently attempted by two groups of researchers - Albrechtsen, Huynh, Jacobs, Knappe, and Wollan \cite{STR23} as well as Georgakopoulos and Papasoglu \cite{GP23}. They both attempted to propose a coarse version of Menger's Theorem using the natural coarse notions of disjoint paths and cuts. We restate their conjecture as: there exists an increasing function $f:\mathbb{N}\to\mathbb{N}$ such that for any graph $G$ with subgraphs $X$ and $Y$ the maximum number of $d$-disjoint $X,Y$ paths is equal to the minimum number of vertices in a set $S$ such that every $X,Y$ path must intersect the $f(d)$-neighborhood of $S$. However, recently a counterexample was discovered that shows this conjecture does not hold in the general case \cite{NSS24}. 

In our recent work \cite{BMS24} we introduced bottlenecking and its coarse counterpart as a measure of ``connectedness" in graphs. We use this notion to quantify when the techniques developed in that work will preserve a quasi-isometric equivalence. Bottlenecking in the classical graph theory setting is related to, but different from connectivity. Several results about bottlenecking in graphs follow as immediate consequences of Menger's Theorem and so our notion of coarse bottlenecking provides a different perspective on coarsening Menger's Theorem. Our formulation of a coarse Mengar conjecture differs from that previously proposed in ways that are natural when it is approached from the prospective of bottlenecking. We have the additional assumption that the sets $X$ and $Y$ must be individually connected and disjoint from each other. This conjecture framed in similar language as the conjectures of \cite{GP23} and \cite{STR23} is for $d=2l$, any graph $G$ the maximum number of $d$-disjoint $X,Y$ paths for any choice of connected subgraphs $X$ and $Y$ such that $d_G(X,Y)>d$ is equal to the minimum $n$ such that for any two connected subgraphs $X'$ and $Y'$ with $d_G(X',Y')>d$ there exists $S\subset V(G)$ with $|S|\leq n$ such that every $X',Y'$ path must intersect the $l$-neighborhood of $S$. 
In Section \ref{Section:CMC'} we give a proof of this conjecture for the class of coarsely bottlenecked graphs. 
\begin{Theorem*}\textbf{(Theorem \ref{CBN_FatMinorClosed})}
 A coarsely bottlenecked graph $G$ is $M$-fat $n$-bottlenecked if and only if $G$ does not contain a $2M$-fat $D_{n+1}$ minor. 
\end{Theorem*} 

In the non-coarse setting this follows as a direct result of Menger's Theorem. In the coarse setting the reverse direction of this is not hard to see. The maximum number of $2M$-disjoint paths between $2M$ disjoint sets is clearly a lower bound on the minimum number of vertices in a set $S$ such that all $X,Y$ paths intersect the $M$-neighborhood of $S$. The forward direction of this implication is more challenging. To give the forward implication we prove what we call a ``Menger-type" Theorem.

\begin{Theorem*}{\textbf{(Theorem \ref{Theorem:CMT}, A Coarse Menger-type Theorem)}}
       If a graph $G$ is coarsely $(n+1)$-bottlenecked but not coarsely $n$-bottlenecked, $G$ contains $D_{n+1}$ as an asymptotic minor.
\end{Theorem*}

In its original form Menger's Theorem was only a statement about finite graphs. Restricting the scope to finite graphs ensures that the maximum number of disjoint $X,Y$ paths, and the number of cuts required to separate $X$ and $Y$ are always finite. This restriction is similar to the assumption we make that our graph must be coarsely bottlenecked. This ensures that, at some scale, our analogous notions are well defined. In 2005 a long-standing conjecture of Erd\H{o}s was resolved by Aharoni and Berger \cite{AB09}. This shows that a natural extension of Menger's Theorem holds for infinite graphs. In hopes that our Coarse Menger-type Theorem may be extended in a similar spirit we propose a Coarse Erd\H{o}s-Menger-type Conjecture.

\begin{Conjecture*} {\textbf{(Conjecture \ref{CEMC} Coarse Erd\H{o}s-Menger-type Conjecture)}}

    If a graph $G$ is not coarsely bottlenecked, it must contain $D_n$ as an asymptotic minor for every $n\in \mathbb{N}$.

    Furthermore, for any cardinality $P$, if a graph is not coarsely $P$-bottlenecked then it contains an asymptotic $D_P$.   
\end{Conjecture*}

To make sense of our Coarse Erd\H{o}s-Menger-type conjecture and make progress towards it, we first must develop an understanding of $P$-bottlenecking for various infinite cardinalities $P$. One possible interpretation of this conjecture would be that, if for a given cardinality $P$, for all $M>0$ there exist connected $M$-disjoint sets $X$ and $Y$ so that for any $S\subset V(G)$ where every $X,Y$ path intersects the $M$-neighborhood of $S$, we have $\mathbf{card}(S)>P$. Then $G$ contains $D_P$ as an asymptotic minor. See \cite{A97} for a discussion on the challenges leading to a complete proof of the classical Erd\H{o}s-Menger Conjecture.

We also propose a conjecture about coarse bottlenecking similar to \cite{GP23} Conjecture 1.1 (reproduced in Conjecture \ref{GP1.1} of this paper). We show that if our Coarse Erd\H{o}s-Menger-type Conjecture holds, then this conjecture is equivalent to a special case of \cite{GP23} Conjecture 1.1.

\begin{Conjecture*}{
\textbf{(Conjecture \ref{Conjectuer:QnBN}
Quasi $n$-Bottlenecking Conjecture)}}

  If a graph $G$ is coarsely $n$-bottlenecked then it is quasi-isometric to a $n$-edge bottlenecked graph.
\end{Conjecture*}

This conjecture holds in the case of $n=1$ and $n=2$. In the case of $n=1$ this follows from Manning's characterization of quasi-trees \cite{M05}. In the case $n=2$ this follows from Fujiwara and Papasoglu's characterization of quasi-cacti \cite{FP23}.

These results motivate the definition of edge-bottlenecking in graph. Graphs that are $1$-edge bottlenecked are trees, and cacti are those that are $2$-edge bottlenecked. In classical graph theory, bottlenecking naturally defines a spectrum of graphs. The first few classes of this spectrum have been studied extensively and have proven to be of interest to a wide audience. In Section \ref{Section:nBN} we extend the existing geometric and excluded minor characterizations of $1$ and $2$-edge bottlenecked graphs to produce a geometric and excluded minor characterization for $3$-bottlenecked graphs. 
\begin{Proposition*}\textbf{(Proposition \ref{Proposition:CC})}
    A graph is $3$-edge bottlenecked iff the intersection of any two cycles is empty or connected.
\end{Proposition*}

This provides a natural view of $3$-edge bottlenecked graphs as a $2$-edge bottlenecked graph with a few edges (or "cuts") added across the cycles. We thus call such a graph a cut-cactus. We further extend the excluded minor characterization to the general case of $n$-edge bottlenecked graphs, showing that the class of $n$-edge bottlenecked graphs is minor closed.

\begin{Theorem*}\textbf{(Theorem \ref{BNMinorClosed})}
The following are equivalent:
\begin{enumerate}
  \item A graph $G$ is $n$-edge bottlenecked.
  \item Every minor of $G$ is $n$-edge bottlenecked.
  \item $G$ does not contain $D_{n+1}$ as a minor.
\end{enumerate}
\end{Theorem*}

\begin{figure}[h]

  \centering
  \resizebox{1.4in}{!}{\includegraphics{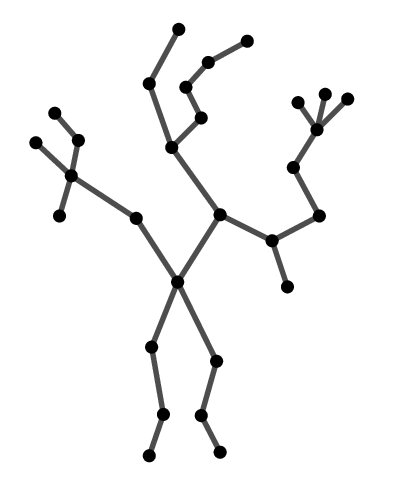}}
  \resizebox{1.4in}{!}{\includegraphics{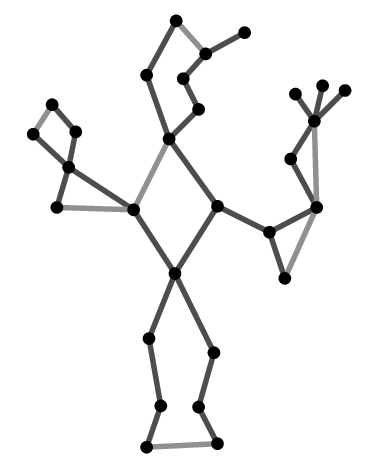}}
  \resizebox{1.4in}{!}{\includegraphics{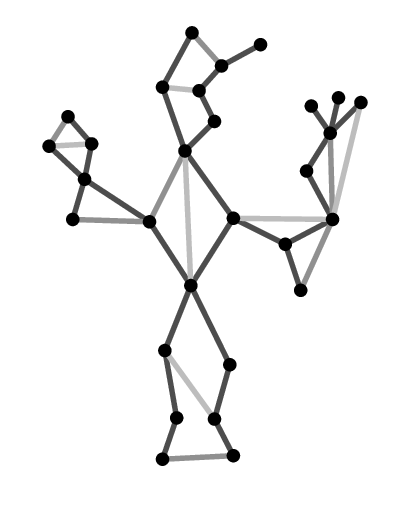}}  \caption{\label{Figure:Tree,Cactus,CutCactus} Examples of a tree, a cactus, and a cut-cactus. Each is produced from the last by the addition of a few edges. Note that while each of these graphs has a different bottlenecking number they are all 1-edge connected.}
\end{figure}

\section{Preliminaries} 
In this work, unless otherwise mentioned, we consider only connected, infinite and unbounded graphs. In this section we reproduce definitions for many common notions from graph theory and coarse geometry, for more information one may refer to a standard text \cite{BMGT,Roe}.
\subsection{Bottlenecking/Coarse Bottlenecking in Graphs}

\begin{Definition}\textbf{(Graph, Path, Cycle, $D_n$, Tree, Cactus, Cut-cactus)}
  A graph $G$ is defined to be two sets, a collection of vertices $V(G)$, and a collection of edges $E(G)$. Each edge is associated with two vertices, said to be its endpoints. A path is an ordered collection of edges that join a set of distinct vertices, a path between two sets of vertices $X$ and $Y$ is said to be an $X,Y$ path. A cycle is a pair of paths that share exactly their endpoints. We often may refer to trees, cactus and cut-cactus these are classes of graphs with restrictions on the interactions of cycles. Trees have no cycles, cactus allow cycles to only intersect at a single vertex, and cut-cactus allow cycles to only intersect along a path. We often refer to a dipole graph $D_n$, this is two vertices joined by $n$ multi-edges.
\end{Definition}
\begin{Definition}
\textbf{($n$-edge bottlenecking)}(See \cite{BMS24})\label{nEdgeB}
  A graph $G$ is said to have $n$-edge bottlenecking if for any two disjoint connected sets $X,Y\subset V(G)$, there exists a set $S\subseteq E(G)$ of size $n$ such that any $X,Y$ path contains an element of $S$.
\end{Definition}
\begin{Definition}\textbf{($M$-fat/Coarse $n$-bottlenecking)}(See \cite{BMS24})\label{mFnB}
  A graph $G$ is said to have $M$-fat $n$-bottlenecking for some $M,n\in\mathbb{N}$ if for any two connected, $M$-disjoint subgraphs $X,Y\subset G$, there exists a set $S\subset V(G)\setminus (V(X)\cup V(Y))$ of size $n$ such that every $X,Y$ path intersects $N_M(S):=\{x\in V(G)|d(x,S)<M\}$. If a graph is $M$-fat $n$-bottlenecked for some $M$ it is said to be coarsely $n$-bottlenecked. If there exists an $n$ such that a graph is coarsely $n$-bottlenecked, the graph is said to be coarsely bottlenecked.
\end{Definition}
\begin{Remark}
  Bottlenecking defines a natural spectrum of graphs as $n$-edge bottlenecking implies $(n+1)$-edge bottlenecking. This is also true for coarse bottlenecking.

\end{Remark}
\begin{Definition}\textbf{(n-ladder/$M$-fat ladder)}
  A graph is said to contain a $n$-ladder if there exist disjoint connected sets $X$ and $Y$ such that there are $n$ $X,Y$ paths that are disjoint outside of $X$ and $Y$. The sets $X$ and $Y$ are referred to as the poles, and the paths are referred to as the rungs of the ladder. If the poles are at least $M$-disjoint and the rungs are all pairwise $M$-disjoint a ladder is said to be a $M$-fat ladder.
\end{Definition}
\begin{Remark}
\label{Remark:BN=Ladder}
  A graph is $n$-edge bottlenecked iff it contains no $(n+1)$-ladders.

  This result follows from Menger's Theorem \cite{Men27} (see Remark \ref{MengerTheorem}).
\end{Remark}

\begin{Definition}\textbf{(Minor, $M$-fat/Asymptotic Minor)} (See \cite{GP23,CDNRV2010})
  A graph $G$ is said to contain a graph $H$ as a minor if there exist a subgraph of $G$ that is the union of disjoint sets corresponding to each of the vertices and edges of $H$, such that contracting each set produces the graph $H$. If these sets are $M$-disjoint except where they are incident in $H$ then this minor is said to be a $M$-fat minor. If a graph contains a $M$-fat $H$ minor for any $M\in\mathbb{N}$ then that graph is said to contain $H$ as an asymptotic minor.
\end{Definition}
\begin{Remark}
    A $n$-ladder is a $D_n$ minor, a $M$-fat $n$-ladder is a $M$-fat $D_n$ minor. If $K$ is a minor of $H$, and $H$ is a $M$-fat minor of $G$, then $K$ is a $M$-fat minor of $G$.
\end{Remark}

\subsection{Other Measures of Connectedness}
\label{Section:ConnectivityVsBN}
\begin{Remark}
  Bottlenecking is closely related to an existing and well understood notion ``Connectivity". Informally bottlenecking measures how connected a graph is by finding the two sets with the highest number of paths between them. Whereas ``Connectivity" measures how connected a graph is by finding the two points with the fewest number of paths between them. 
\end{Remark}
\begin{Definition}\textbf{($n$-Edge Connectivity)}
  A graph $G$ is said to have $n$-edge connectivity iff for any two vertices $x,y\in V(G)$ there exists a set of at least $n$ number of $x,y$ paths that share no edges.
\end{Definition}
\begin{Proposition}\label{Proposition:ConnectivityImpliesNotBottlenecked}
  If a graph is $n$-edge connected then it is not $(n-1)$-edge bottlenecked.
\end{Proposition}
\begin{proof}
  If a graph is $n$-edge connected then any two vertices have at least $n$ edge disjoint paths between them so there is a pair of two vertices $x,y$ such that no set of $n-1$ edges intersects every $x,y$ path.
\end{proof}

\begin{Remark}
  The converse of Proposition \ref{Proposition:ConnectivityImpliesNotBottlenecked}
  does not hold showing that connectivity is a stronger property than bottlenecking (see Figure \ref{Figure:Tree,Cactus,CutCactus}).
\end{Remark}

\begin{Remark}
  Similarly to edge bottlenecking and edge connectivity there exist notions of point bottlenecking and point connectivity.
\end{Remark}

\begin{Definition}\textbf{($n$-point bottlenecking)}\label{nPointB}
  A graph $G$ is said to have $n$-point bottlenecking if for any two disjoint connected sets $X,Y\subset V(G)$, there exists a set $S\subseteq V(G)$ of size $n$ such that any $X,Y$ path contains an element of $S$.
\end{Definition}

\begin{Definition}\textbf{($n$-point Connectivity)}
  A graph $G$ is said to have $n$-point connectivity iff for any two vertices $x,y\in V(G)$ there exists a set of at least $n$ internally disjoint $x,y$ paths.
\end{Definition}

\begin{Proposition}
  If a graph is quasi-isometric to a $n$-edge bottlenecked graph then it is quasi-isometric to a $n$-point bottlenecked graph.
\end{Proposition}
\begin{proof}
  If a graph is $n$-edge bottlenecked then by subdividing every edge once a graph with $n$-point bottlenecking may be produced.
\end{proof}
\begin{Remark}
  It is not clear if there are $n$-point bottlenecked graphs that are not quasi-isometric to a edge bottlenecked graph. However both notions coarsen in the same way.
\end{Remark}

\begin{Remark}\textbf{({Menger's Theorem})}
\label{MengerTheorem}
  The edge (or vertex) connectivity of a graph is equal to the minimum number of edges (or vertices) that must be removed from a graph to disconnect it.
\end{Remark}
\subsection{Searching for a coarse Menger-type theorem}

In our attempts to coarsen Menger's theorem, we analyse the counterexample in its originally proposed coarse version to see where the conjecture fails.

\begin{Example}\textbf{(Nguyen,  Scott \& Seymour, \cite{NSS24})}
\label{Example: counterexample}
    The  counterexample \cite{NSS24} to the coarse Menger's conjecture proposed in \cite{STR23} and \cite{GP23} consists of a $(2l+2)$-deep binary tree where the outermost branches are $2l$ long and connected at their ends with a similarly long path. They identify two overlapping sets $X,Y$ where there are at most two $3$-disjoint paths but no set of two vertices has a $l$-neighborhood that intersects every $X,Y$ path. Figure \ref{Figure:Counterexample} shows this construction. 
    
    This counterexample relies on the result in their construction that for any two $X,Y$ paths, either one pass through the root $r$ or they are at most $2$-disjoint. This result stems from the fact that the root is a single point along with the way the leaves overlap in the middle of the connecting path: All $X,Y$ paths must use $r,q_1$ or $q_2$ as shown in the Figure \ref{Figure:Counterexample}. For this counterexample to work, both sets $X$ and $Y$ must contain the root. Enforcing coarse disjointedness on these sets breaks the counterexample in such a way that it appears difficult to adapt into a similar counterexample. Based on this analysis, we present in section \ref{Section:CMC'} a coarse Menger-type Theorem \ref{Theorem:CMT} with a disjointedness requirement among some other restrictions that follow from coarse bottlenecking.
\end{Example}

\begin{figure}
    \centering
    \resizebox{5in}{!}{\includegraphics[width=0.5\linewidth]{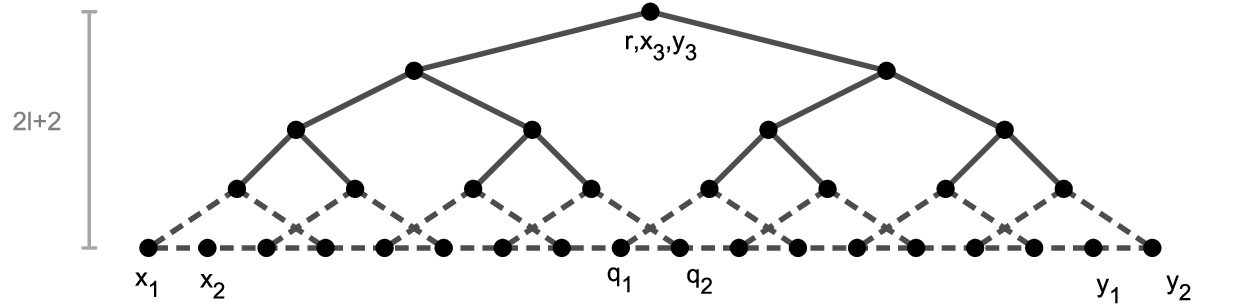}}
    \caption{Construction of the counterexample (reproduced from \cite{NSS24}). A binary tree of depth $2l+2$ from root point $r$. The dashed lines represent paths of length $2l$. Any path from $x_1,x_2,x_3 \in X$ to $y_1,y_2,y_3 \in Y$ must either be $r$ or contain point $q_1$ or $q_2$, so any pair of $X,Y$-paths that do not contain $r$ are at most $2$-disjoint. Note that for the counter example to work $X$ and $Y$ must both contain the root point $r$, and hence cannot be disjoint.}
    \label{Figure:Counterexample}
\end{figure}

\section{Characterizations of $n$-edge bottlenecked graphs} 
\label{Section:nBN}
\begin{Remark}
  The class of connected graphs that are $1$-edge bottlenecked has been studied extensively under the title ``Trees" and has proven to be of interest in many different fields. A characterization similar to what one might find in any textbook on graph theory is given here with the added equivalence to $1$-edge bottlenecking. 
\end{Remark}
\subsection{Trees}

\begin{Proposition}
  The following are equivalent
  \begin{enumerate}
    \item A graph is $1$-edge bottlenecked.
    \item A graph does not contain $D_2$ as a minor.  \item A graph has a unique path between any two vertices.
    \item A graph contains no cycles.
  \end{enumerate}
\end{Proposition}
\begin{proof}[$1\implies 2$]

Assume to a contradiction a graph contains a $D_2$ minor, then the sets that contract to form the vertices of $D_2$ give sets $X$ and $Y$ such that there is no edge intersecting every $X,Y$ path, contradicting condition $1$. 
\end{proof}
\begin{proof}[$2\implies 3$]
Assume to a contradiction that there were two distinct paths between two vertices, then there is an edge in one path but not the other, by contracting the whole graph to the endpoints of this edge a $D_2$ minor is produced, contradicting condition $2$.
\end{proof}
\begin{proof}[$3\implies 4$]
If a graph has a unique path between any two vertices then there is no pair of internally disjoint paths that share their endpoints.
\end{proof}
\begin{proof}[$4\implies 1$]
 Assume to a contradiction that a graph is not $1$-edge bottlenecked, this gives that there are connected sets $X$ and $Y$ such that there is no edge that intersects every $X,Y$ path. There must be at least two $X,Y$ paths, and there must be at least one edge contained in only one of these two paths. This edge is a path between its endpoints, and as there is an $X,Y$ path not through this edge there must be a second path between its endpoints, this gives a cycle and contradicts condition $4$.
\end{proof}

\subsection{Cacti}
\begin{Remark}
The class of graphs that are $2$-edge bottlenecked is known to many as "Cactus" graphs. These have been studied to some extent as well, but not to the same level as trees.
\end{Remark}
\begin{Proposition}
\label{Proposition:Cactus}
  The following are equivalent
  \begin{enumerate}
    \item A graph is $2$-edge bottlenecked.
    \item A graph does not contain $D_3$ as a minor.
    \item The intersection of any two distinct cycles is at most a single vertex.
  \end{enumerate}
\end{Proposition}
\begin{proof}[$1 \implies 2$]
  If a graph contains a $D_3$ minor then each of the sets that contract to form the edges of $D_3$ give 3 disjoint paths between the sets that contract to form the vertices of $D_3$.
\end{proof}
\begin{proof}[$2 \implies 3$]
  If the intersection of two cycles contains more than one vertex then it either contains an edge or is disconnected. If the intersection is disconnected then Proposition \ref{Proposition:CC} gives a stronger result so we may assume towards a contradiction there are two cycles with an intersection that is connected and contains at least two vertices. By contracting the endpoints of this connected region (that contains at least one edge) a $D_3$ minor is produced, contradicting condition $2$.
\end{proof}
\begin{proof}[$3 \implies 1$]
  Assume to a contradiction that there are connected sets $X$ and $Y$ such that no two edges will intersect every $X,Y$ path. This gives at least three edges with three $X,Y$ paths such that each path uses only one of the edges. As $X$ and $Y$ are connected any pair of these paths gives a cycle, so two pairs of paths give two cycles that share an edge.
\end{proof}

\begin{Proposition}

\label{3disjointarc}
If a graph $G$ contains a $D_3$ minor then it contains a subdivision of $D_3$.
\end{Proposition}
\begin{proof}\label{Lemma 100} 

By Proposition \ref{Proposition:Cactus} a $D_3$ minor implies there are two distinct cycles that intersect at more than one vertex. let $C_1$ and $C_2$ be two such cycles. Cycle $C_1$ contains a path $S$ with only its endpoints in $C_2$. $S \cup C_2$ is a structure that can be produced by subdividing $D_3$.
\end{proof}

\subsection{Cut-Cacti}
\begin{Remark}
  After trees and cactus there is the class of three edge bottlenecked graphs we call cut-cactus. The authors are, at the time of writing, unaware of any in-depth study of such graphs. 
\end{Remark}
\begin{Proposition}
\label{Proposition:CC}
  The following are equivalent
  \begin{enumerate}
    \item A graph is $3$-edge bottlenecked.
    \item The intersection of any two distinct cycles in $G$ is empty or connected.
    \item A graph does not contain $D_4$ as a minor.
  \end{enumerate}
\end{Proposition}
\begin{proof}[$1\implies 2$]
  Let $G$ be a $3$-edge bottlenecked graph with cycles $C_1$ and $C_2$. Assume to a contradiction that $C_1\cap C_2$ is nonempty and not connected. This gives that there are points $x,y\in C_1\cap C_2$ such that no $x,y$ path is contained in $C_1\cap C_2$. Consider the set $X$ to be the connected component of $C_1\cap C_2$ containing $x$. Take the sets $r_1,r_2,r_3,r_4$ to be the segments of $C_1\cup C_2\setminus C_1\cap C_2$ adjoined to $X$. Then take the set $Y$ to be all of $C_1\cup C_2$ aside from $X$ and the $r_i$ sets. The $r_i$ segments give 4 internally disjoint $X,Y$ paths, contradicting condition $1$.
\end{proof}
\begin{proof}[$2\implies 3$]
Let $G$ be a graph where the intersection of any two distinct cycles is empty or connected and assume to a contradiction that $G$ contains a $D_4$ minor. By considering 3 of the edges in this $D_4$ minor there is a $D_3$ minor, and by proposition \ref{3disjointarc} a $D_3$ subdivision. By considering the shortest paths between the vertices of this $D_3$ subdivision that use the remaining edge of the $D_4$ minor, there are two cycles with a disconnected intersection and this is a contradiction.
\end{proof}

\begin{proof}[$3\implies 1$]
If a graph is not $3$-edge bottlenecked then there are sets $X$ and $Y$ with at least 4 paths that share no edges between them. By contracting each path to an edge and the sets $X$ and $Y$ to vertices a $D_4$ minor is produced, contradicting condition $3$.
\end{proof}

\begin{Proposition}
\label{Lemma: 101}
  If a graph contains a $4$-ladder then it contains a $4$-ladder where the poles are paths, and the connections of two rungs occur at each of the endpoints of each path.
\end{Proposition}
\begin{proof}
Let $G$ be a graph that contains a $4$-ladder. There are two cycles in $G$ that have an intersection that is not connected. Let one connected region of this intersection be pole. The cycles give four paths that leave the first pole and eventually connect again at another region of the intersection. Take these segments to be the rungs and tale a path connecting their endpoints to be the other pole. This forms the required $4$-ladder structure.
\end{proof}

\subsection{$n$-edge bottlenecking}

\begin{Theorem} 
\label{BNMinorClosed}
The following are equivalent:
\begin{enumerate}
  \item A graph $G$ is $n$-edge bottlenecked.
  \item Every minor of $G$ is $n$-edge bottlenecked.
  \item $G$ does not contain $D_{n+1}$ as a minor.
\end{enumerate}
\end{Theorem}
  \begin{proof}[$1 \iff 2$]
    The backwards direction is straight-forward as $G$ is a minor of itself. So assume to a contradiction that $G$ is $n$-edge bottlenecked for some $n\in\mathbb{N}$ and that there is some minor $H$ of $G$ that is not $n$-bottlenecked. This gives that there are $X,Y\subset V(H)$ such that there are at least $(n+1)$-edge independent $X,Y$ paths. By considering the sets in $G$ that contract to form $X,Y$ and these paths one can construct an $(n+1)$-ladder in $G$. This would imply that $G$ is not $n$-edge bottlenecked, a contradiction.
    \end{proof}
  \begin{proof}[$1 \iff 3$]
  If $G$ contains $D_{n+1}$ as a minor it is clear that $G$ is not $n$-edge bottlenecked. So let $G$ be a graph that does not contain a $D_{n+1}$ minor and assume to a contradiction that $G$ is not $n$-bottlenecked. As $G$ is not $n$-edge bottlenecked there are $X,Y\subset V(G)$ with $(n+1)$-edge independent paths between them. By picking the first edge on each path that has an endpoint not in $X$ and contracting all of $X,Y$ allong with the rest of these paths to points, this gives a $D_{n+1}$ minor. 
  \end{proof}

 \begin{Remark}
 Theorem \ref{BNMinorClosed} shows that the class of $n$-bottlenecked graphs is a minor closed family classified by the $D_{n+1}$ excluded minor.   
 \end{Remark}
 \begin{Remark}
     Propositions \ref{Proposition:Cactus} and \ref{Proposition:CC} provide a geometric viewpoint to help understand the way cycles interact in $2$ and $3$-edge bottlenecked graphs.  It would be  interesting to extend this geometric characterization across the whole spectrum of edge bottlenecked graphs.
 \end{Remark}

\section{Coarse Bottlenecking \& A
Coarse Menger-type Theorem}
\label{Section:CMC'}
\begin{Remark}
  Remark \ref{Remark:BN=Ladder} Gives ladders as a useful tool for studying edge bottlenecking. A $2M$-fat $(n+1)$-ladder is not $M$-fat $n$-bottlenecked. The converse of this statement would be a slightly different version of the coarse Menger Conjecture than previously proposed \cite{GP23,CDNRV2010} and disproven by a counterexample \cite{STR23}. Note that the counter example given does not apply in this case as it contains a fat $3$-ladder. We present one such statement below.
\end{Remark}
\begin{Theorem}
\label{CBN_FatMinorClosed}
  Consider the following conditions:
  \begin{enumerate}
      \item A graph $G$ is $M$-fat $n$-bottlenecked. 
       \item Every $2M$-fat minor of $G$ is $n$-edge bottlenecked. 
      \item $G$ does not contain a $2M$-fat $D_{n+1}$ minor.
      \item $G$ does not contain a $2M$-fat $n+1$-ladder.
  \end{enumerate}
    The implications $1\implies 2\iff 3\iff 4$ hold. Furthermore if the graph is coarsely bottlenecked then $4\implies 1$.  
\end{Theorem} 
\begin{proof}[Proof of Theorem \ref{CBN_FatMinorClosed}]

The implication $3\iff 4$ follows directly from the definitions of a $M$-fat ladder and a $M$-fat minor.
The implication $4\implies 1$ holds for coarsely bottlenecked graphs as a consequence of Theorem \ref{Theorem:CMT}. The implications $2\iff 3$ and $1\implies 4$ are shown in Lemma \ref{Lemma:2=3} and Lemma \ref{Lemma:1=>4} respectively.
\end{proof}
\begin{Lemma}[$2\iff 3$]\label{Lemma:2=3}
Every $2M$-fat minor of $G$ is $n$-bottlenecked iff $G$ does not contain a $2M$-fat $D_{n+1}$ minor.
\end{Lemma}
\begin{proof}
    Let $G$ be a graph such that every $2M$-fat minor of $G$ is $n$-bottlenecked. As $D_{n+1}$ is not $n$-bottlenecked $G$ does not contain a $2M$-fat $D_{n+1}$ minor. 

    Let $G$ be a graph that does not contain a $2M$-fat $D_{n+1}$ minor. Let $H$ be a $2M$-fat minor of $G$,  every minor of $H$ is a $2M$-fat minor of $G$. Therefore $H$ must not contain a $D_{n+1}$ minor and so $H$ must be $n$-bottlenecked.
\end{proof}

\begin{Lemma}[$1\implies 4$]\label{Lemma:1=>4}
   If a graph $G$ contains a $2M$-fat $(n+1)$-ladder then $G$ is not $M$-fat $n$-bottlenecked.
\end{Lemma}
\begin{proof}
     Let $G$ be a graph that contains a $2M$-fat $(n+1)$-ladder. The poles give sets $X$ and $Y$ with $n+1$ distinct paths between them such that any two paths are $2M$-disjoint. As the paths are at least $2M$-disjoint, any $M$-neighborhood around a vertex will intersect at most one path. Therefore any set $S\subset V(G)\setminus (V(X)\cup V(Y))$ such that $N_M(S)$ intersects every path must have $|S|>n$ showing that $G$ is not $M$-fat $n$-bottlenecked. 
\end{proof}

The analog in a non-coarse setting of the implication $4\implies 1$ is given in Remark \ref{Remark:BN=Ladder} and follows as a direct result of Menger's Theorem \cite{Men27}. It has been shown \cite{STR23} that a coarse Menger's Theorem (as proposed by \cite{GP23,STR23}) does not hold. Bottlenecking is related to, but weaker than connectivity; so, there is reason to hope that a similar statement pertaining to coarse bottlenecking may be true. The existing counter example relies on large binary trees with connections across the lower layers, and as these structures get big they are not coarsely bottlenecked. See Example \ref{Example: counterexample} for more discussion on the differences between our coarse Menger-type Theorem and prior attempts at a Coarse Menger conjecture.

In a prior version of this paper \cite{V1}, the authors proposed a ``Coarse Menger-type Conjecture" to serve as the implication $4\implies 1$ for Theorem \ref{CBN_FatMinorClosed}. Under the assumption that the graph is coarsely bottlenecked, Theorem \ref{Theorem:CMT} proves our original conjecture. 

\begin{Theorem}[\textbf{A  coarse Menger-type Theorem}]
\label{Theorem:CMT}
    If a graph $G$ is coarsely $n+1$-bottlenecked but not coarsely $n$-bottlenecked, $G$ contains $D_{n+1}$ as an asymptotic minor.
\end{Theorem}

\begin{proof}
    
    See Figure \ref{figure:CMT}. We will show $G$ contains $D_{n+1}$ as an asymptotic minor by constructing a $M$-fat $n+1$-ladder in $G$ for arbitrarily large $M$.

    See Figure \ref{figure:CMT}. Let $M>1$, as $G$ is not coarsely $n$-bottlenecked, for $B>2(M+m+1)$, there exist connected $B$-disjoint sets $X$ and $Y$ such that no set of $n$ vertices has a $B$-neighborhood that will intersect every $X,Y$ path. However, as $G$ is $m$-fat $n+1$-bottlenecked there is a collection of $n+1$ vertices $R:=r_1,r_2...r_{n+1}$ such that $N_m(R)$ intersects every $X,Y$ path. 
    
    Note that $d(r_i,r_j) > 2(B-m)$ for all $i\neq j$, else a $B$-neighborhood around the midpoint of $r_i,r_j$ would enclose $N_m(r_j)$ and $N_m(r_i)$, contradicting the result a $B$-neighborhood around $n$ vertices cannot intersect every $X,Y$ path. As the $r_i$ are each pairwise $2(B-m)$-disjoint, $m+2M+1$-neighborhoods around them are disjoint so $N_{m+M+1}(R)$ consists of $n+1$ pairwise $M$-disjoint components. These will form the rungs of our ladder.
    
    Next we will construct sets $X_P$ and $Y_P$ to form the poles. The set $X_P$ is the union of the set $X$ and each component of $G\setminus N_{m+M}(R)$ that contains a part of $X$. This set is connected as the set $X$ is connected. The set $Y_P$ is constructed similarly.
      
    As $X$ and $Y$ are disconnected in $G\setminus N_{m}(R)$, the components of $G\setminus N_{m+M}(R)$ that intersect $X$ are $2M$-disjoint from the components intersecting $Y$. Furthermore, as every $X,Y$ path intersects $N_{m}(R)$, the set $X$ is $M$-disjoint from components of $N_{m+M}(R)$ that intersect $Y$. A symmetric argument shows that the sets $X_P$ and $Y_P$ are $M$-disjoint.

    As the poles and rungs are $M$-disjoint, these sets give a $M$-fat $(n+1)$-ladder. Now the result follows from the result of Theorem \ref{CBN_FatMinorClosed} that shows the equivalence between fat ladders and fat $D_n$ minors.
\end{proof}

Menger's Theorem was first proven for finite graphs. In 2005 this was extended to infinite graphs by Aharoni and Berger \cite{AB09} who proved a conjecture of Erd\H{o}s (The Erd\H{o}s-Menger Conjecture). Menger's Theorem required the graph to be finite to ensure that the maximum number of disjoint paths between two sets was finite. This is similar to how we require the graph to be coarsely bottlenecked to ensure that the maximum number of coarsely disjoint paths is finite. We provide the following conjecture that would extend our Menger-type theorem in the same way as the Erd\H{o}s-Menger Theorem extended Menger's Theorem.

\begin{Conjecture}\textbf{(A coarse Erd\H{o}s-Menger-type Conjecture)}
\label{CEMC}
  If a graph is not coarsely bottlenecked, it must contain $D_n$ as an asymptotic minor for every $n\in \mathbb{N}$.
\end{Conjecture}

In the introduction we provide an extended formulation of this conjecture that more closely parallels the Erd\H{o}s Menger Theorem.

\begin{figure}[h]
    \centering
    \resizebox{5in}{!}{\includegraphics{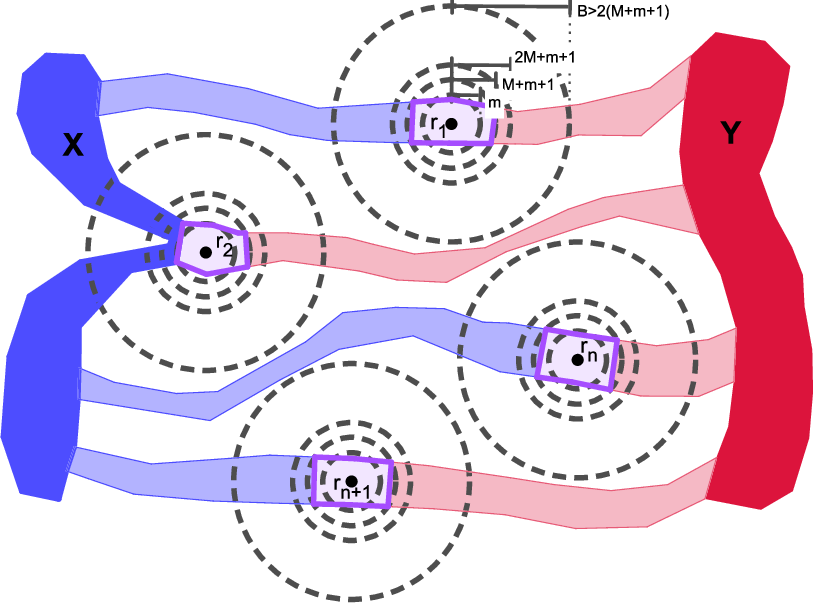}}
    \caption{\textbf{Picture proof of theorem \ref{Theorem:CMT}}  The set $X$ is on the left in dark blue, while the set $Y$,  on the right dark red. All $X,Y$ paths are in the other colored regions. Components of $G \setminus N_{m+M}(R)$ are light blue if they touch $X$ or light red if they touch $Y$. By  construction, these components are $M$-disjoint from $N_{m}(R)$. $X$ and $Y$ can enter an $m$-neighborhood, but never the same one nor cross through $N_{m}(R)$ on an $X,Y$ path, as this would lead to a contradiction. The poles $X_p$ (everything blue), and $Y_p$ (everything red) are $M$-disjoint as they are separated by $M$ long portions of the $X,Y$ paths between the $m$-neighborhoods and the $M+m$-neighborhoods. The $m+2M+1$-neighborhoods  for each vertex do not overlap as this would contradict $B> 2(M+m+1)$, so the rungs (in purple) are formed by the components of $N_{m+M+1}(R)$ and are $M$-disjoint from each other.}
    \label{figure:CMT}
\end{figure}

\begin{Remark}
  Theorem \ref{Theorem:CMT} shows that, for a coarsely bottlenecked graph. Coarse $n$-bottlenecking is equivalent to excluding $D_{n+1}$ as an asymptotic minor. It does this by allowing the construction of a $M$-fat $n'$-ladder for any $M$ and all $n'\leq n+1$.
 
 Together with Theorem \ref{CBN_FatMinorClosed}, this shows that the class of $M$-fat $n$-bottlenecked graphs is a minor closed family classified by the $M$-fat $D_{n+1}$ excluded minor. 
 \end{Remark}

\begin{Theorem}
\label{Theorem:CNB}
  Consider the conditions:
  \begin{enumerate}
  \item A graph $G$ is not coarsely $n$-bottlenecked.
  \item $G$ contains $D_{n+1}$ as an asymptotic minor.
  \item $G$ contains a $(n+1)$-ladder as an asymptotic minor.
\end{enumerate}
The implications $2\iff 3\implies 1$ hold. Furthermore if the graph is coarsely bottlenecked then $1\implies 2$.
\end{Theorem}
\begin{proof}[Proof of Theorem \ref{Theorem:CNB}]
    The implications $2\iff 3\implies 1$ hold as a result of Theorem \ref{CBN_FatMinorClosed}. In coarsely bottlenecked graphs $1\implies 2$ as a result of Theorem \ref{Theorem:CMT}. In cases where $1\implies 3$ holds, this gives an excluded asymptotic minor characterization of coarsely bottlenecked graphs.
\end{proof}
In \cite{GP23} Georgakopoulos and Papasoglu proposed the following conjecture.
\begin{Conjecture}[\cite{GP23} Conjecture 1.1]
\label{GP1.1}
Let $X$ be a graph or a length space, and let $H$ be a finite
graph. Then $X$ has no $M$-fat $H$ minor for some $M \in \mathbb{N}$ if and only if $X$ is quasi-isometric to a graph with no $H$ minor. Furthermore, the constants of this quasi-isometry depend only on $M$ and $H$.
\end{Conjecture}
\begin{Remark}
  Recently  a counter example \cite{DHIM24} to the coarse Menger conjecture of \cite{GP23} and \cite{STR23} was adapted to a counterexample of Conjecture \ref{GP1.1}. In a similar spirit, we propose the following conjecture for the class of coarsely bottlenecked graphs.
\end{Remark}
\begin{Conjecture}[\textbf{Quasi $n$-Bottlenecking Conjecture}]
\label{Conjectuer:QnBN}
  If a graph $G$ is coarsely $n$-bottlenecked then it is quasi-isometric to a $n$-edge bottlenecked graph.
\end{Conjecture}
\begin{Remark}
  The Quasi-Bottlenecking conjecture \ref{Conjectuer:QnBN} holds in the case of $n=1$ \cite{M05,GP23} and $n=2$ \cite{FP23,STR23}.
  \end{Remark}
  \begin{Proposition}
      If our Coarse Menger-type Conjecture (\ref{CEMC}) holds, then conjecture \ref{Conjectuer:QnBN} is a special case of Conjecture 1.1 of Georgakopoulos and Papasoglu (Conjecture \ref{GP1.1} in this work).
  \end{Proposition}
  \begin{proof}
      Conjecture \ref{CEMC} states that a graph that is not coarsely bottlenecked will contain an asymptotic $D_n$ minor for all $n$. This, in combination with Theorem \ref{Theorem:CMT} would imply that coarse $n$-bottlenecking is characterized by an excluded asymptotic $D_{n+1}$ minor. In this case the assumption of Conjecture \ref{Conjectuer:QnBN} would be equivalent to the exclusion of an asymptotic $D_{n+1}$ minor, and this is exactly a special case of Conjecture \ref{GP1.1}.
  \end{proof}

\end{document}